\documentclass[a4paper,12pt]{article}

\usepackage{epsfig,amsmath,amsfonts,amssymb,mathtools,amsthm,graphics,xypic}
\usepackage{mathrsfs,enumerate}
\usepackage[usenames,dvipsnames]{color}

\def\ga{\gamma}
\def\Ga{{\Gamma}}
\def\de{\delta}

\def\De{\Delta}
\def\eps{{\varepsilon}}

\def\om{\omega}

\def\Om{\Omega}

\def\sig{{\sigma}}

\def\th{{\theta}}

\newcommand{\ph}{\varphi}

\def\ze{{\zeta}}

\newcommand{\demi}{\frac{1}{2}}

\newcommand{\ov}{\overline}

\newcommand{\dist}{\operatorname{dist}}

\newcommand{\cont}{\operatorname{cont}}

\newcommand{\ID}{\mathop{\hbox{{\rm Id}}}\nolimits}

\newcommand{\I}{{\mathrm i}}
\newcommand{\dd}{{\mathrm d}}

\newcommand{\ee}{\mathrm e}

\newcommand{\ii}{^{-1}}

\newcommand{\ti}{\tilde}

\newcommand{\eopf}{
\hfill\begin{picture}(.24,.24)\thinlines
\put(0,0){\line(1,0){.24}}
\put(.24,0){\line(0,1){.24}}
\put(.24,.24){\line(-1,0){.24}}
\put(0,.24){\line(0,-1){.24}}
\end{picture}\bigskip }

\def\ie{{\it i.e.}\ }

\def\eg{{\it e.g.}\ }
\def\resp{{resp.}\ }
\def\wrt{{with respect to}}

\def\rhs{{right-hand side}}

\newcommand{\C}{\mathbb{C}}

\newcommand{\D}{\mathbb{D}}

\newcommand{\N}{\mathbb{N}}

\newcommand{\R}{\mathbb{R}}

\newcommand{\Z}{\mathbb{Z}}

\def\cF{\mathcal{F}}

\def\cL{\mathcal{L}}

\newtheorem{thm}{Theorem}[section]

\newtheorem{lemma}[thm]{Lemma}

\newtheorem{Def}[thm]{Definition}

\newtheorem{nota}[thm]{Notation}

\theoremstyle{definition}
\newtheorem{rem}[thm]{Remark}
\newtheorem{exa}[thm]{Example}

\newcounter{parag}[section]

\newcounter{parage}

\newcounter{paraga}

\def\om{\omega}

\def\Bbibitem#1#2{\bibitem[#1]{#2}}

\newcommand{\defeq}{\coloneqq} 

\newcommand{\col}{\colon\thinspace}          

\newcommand{\gR}{\mathscr R}       

\newcommand{\Ddem}{\D_{\rho/2}}

\newcommand{\beglabel}[1]{\begin{equation}	\label{#1}}
\newcommand{\elabel}{\end{equation}}

\begin{document}




\title{On the stability under convolution of \\ resurgent functions}

\author{David Sauzin}


\maketitle

\vspace{-.75cm}

\begin{abstract}
This article introduces, for any closed discrete subset~$\Om$ of~$\C$, the
definition of $\Om$-continuability, 
a particular case of \'Ecalle's resurgence:
$\Om$-continuable functions are required to be holomorphic near~$0$ and to admit
analytic continuation along any path which avoids~$\Om$.
We give a rigorous and self-contained treatment of the stability under
convolution of this space of functions,
showing that a necessary and sufficient condition is the stability of~$\Om$
under addition.

\medskip

\end{abstract}


\section{Introduction}




\'Ecalle's theory of resurgent functions is an efficient tool for dealing with
divergent series arising from complex dynamical systems or WKB expansions, and
for determining the analytic invariants of differential or difference equations.
Fundamental notions of the theory are that of 
\emph{germs analytically continuable without a cut},
and the related notion of
\emph{endlessly continuable germs}: these are holomorphic germs of one complex
variable at the origin which enjoy a certain property of analytic continuation
(the possible singularities of their analytic continuation must be isolated, at
least locally---\cite{Eca81}, \cite{Mal85}, \cite{CNP}); they arise as Borel
transforms of possibly divergent formal series which solve certain nonlinear
problems.

Since the theory is designed to deal with nonlinear problems, it is an
essential fact that the property of endless continuability 
(or of continuability without a cut)
is stable under convolution (indeed, via Borel transform, the convolution of
germs at~$0$ reflects the Cauchy product of formal series).
This allows to define the \emph{algebra} of resurgent functions in the
``convolutive model'' and then to study certain subalgebras obtained by
specifying the location or the nature of the possible singularities that one can
encounter in the process of analytic continuation.
\'Ecalle then proceeds with defining the ``alien calculus'', which involves
particular derivations of this algebra and is an efficient way of encoding the
singularities, and deriving consequences in the ``geometric models'' obtained by
applying the Laplace transform in all possible directions; 
this is a way of describing nonlinear Stokes phenomena or of solving
problems of analytic classification---see \cite{Eca81},
\cite{dulac}, \cite{Eca93}, \cite{CNP}, \cite{kokyu}, \cite{mouldSN}, \cite{Lima}.

Unfortunately, the proof of the stability under convolution of endlessly
continuable germs in full generality is difficult. \'Ecalle's argument is based
on the notion of ``symmetrically contractile'' paths, but the fact that one can
always find such paths is a delicate matter.
Therefore, when we came across a strikingly simple proof which applies to
interesting subspaces of resurgent functions, we thought it was worthwhile to
bring it to the attention of researchers interested in resurgence theory.

We shall deal in this article with a particular case of endless continuability,
which we call $\Om$-continuability, which corresponds to specifying a priori the
possible location of the singularities: they are required to lie in a set~$\Om$
that we fix in advance. This means that there is one Riemann surface over~$\C$,
depending only on~$\Om$, on which every $\Om$-continuable germ induces a
holomorphic function
(whereas in the general case of endless continuability there is an ``endless''
Riemann surface which does depend on the considered germ).
This definition already covers interesting cases:
one encounters $\Om$-continuable germs with $\Om=\N^*$ or $\Om=\Z$
when dealing with differential equations formally conjugate to the Euler
equation (in the study of the saddle-node singularities) \cite{Eca84}, \cite{mouldSN},
or with $\Om=2\pi\I\Z$ when dealing with certain difference equations like Abel's
equation for parabolic germs in holomorphic dynamics \cite{Eca81}, \cite{kokyu},
\cite{EVinv}, \cite{Lima}.




Our aim is to give a rigorous and self-contained treatment of the stability
under convolution of the space of $\Om$-continuable germs,
with more details and more complete explanations than \eg
\cite{kokyu} which was dealing with the particular case $\Om=2\pi\I\Z$.
For the latter case, the recent article \cite{Y_Ou} is available, but our
approach is different.

For any closed discrete subset of~$\C$, we shall thus 
introduce the definition of $\Om$-continuability in Section~\ref{secOmcont},
recall the definition of convolution in Section~\ref{sec_contconvoleasy}
and state in Section~\ref{sec_contconvolgen} our main result,
Theorem~\ref{thmOmstbgROmstb}, which is \emph{the equivalence of the stability
under convolution of $\Om$-continuable germs and the stability under addition of
the set~$\Om$}.
The rest of the article will be devoted to the proof of this theorem.

A novel feature of our proof (even if we certainly owe a debt to \cite{Eca81}
and \cite{CNP}) is the construction of ``symmetric $\Om$-homotopies'' by means
of certain non-autonomous vector fields.


\section{The $\Om$-continuable germs}	\label{secOmcont}






In this article, ``path'' means a piecewise $C^1$ function $\ga \col
J\to\C$, where~$J$ is a compact interval of~$\R$.
For any $R>0$ and $\ze_0\in\C$ we use the notations 
$D(\ze_0,R) \defeq \{\, \ze\in\C \mid |\ze-\ze_0|< R \,\}$,
$\D_R \defeq D(0,R)$ and
$\D^*_R \defeq \D_R \setminus \{0\}$.


\begin{Def}	\label{DefOmCont}
Let $\Om$ be a non-empty closed discrete subset of~$\C$,
let $\hat\ph(\ze)\in\C\{\ze\}$ be a holomorphic germ at the origin.
We say that $\hat\ph$ is $\Om$-continuable if there exists $R>0$ not larger than
the radius of convergence of~$\hat\ph$ such that $\D^*_R\cap\Om=\emptyset$
and $\hat\ph$ admits analytic continuation along any path of $\C\setminus\Om$
originating from any point of~$\D^*_R$.
We use the notation
\[
\hat\gR_\Om \defeq \{\, \text{all $\Om$-continuable holomorphic germs} \,\}
\subset \C\{\ze\}.
\]
\end{Def}



\begin{rem}	\label{reminitialpoint}
Let $\rho \defeq \min\big\{ |\om|, \; \om\in\Om\setminus\{0\} \big\}$.
Any $\hat\ph\in\hat\gR_\Om$ is a holomorphic germ at~$0$ with radius of
convergence $\ge\rho$ and one can always take $R=\rho$ in
Definition~\ref{DefOmCont}. 
In fact, given an arbitrary $\ze_0 \in \D_\rho$, we have
\[
\hat\ph\in\hat\gR_\Om \quad\Longleftrightarrow\quad
\left| \begin{aligned}
&\text{$\hat\ph$ germ of holomorphic function of~$\D_\rho$ admitting analytic}\\
&\text{continuation along any path $\ga\col [0,1] \to \C$ such that} \\
&\text{$\ga(0) = \ze_0$ and $\ga\big( (0,1] \big) \subset \C\setminus\Om$}
\end{aligned} \right.
\]
(even if $\ze_0 = 0$ and $0\in\Om$: there is no need to avoid~$0$ at the beginning of the
path, when we still are in the disc of convergence of~$\hat\ph$).
\end{rem}
%




\begin{exa}
Trivially, any entire function of~$\C$ defines an $\Om$-continuable germ.
Other elementary examples of $\Om$-continuable germs are the functions which are
holomorphic in $\C\setminus\Om$ and regular at~$0$, like $\frac{1}{(\ze-\om)^m}$
with $m\in\N^*$ and $\om\in\Om\setminus\{0\}$.
But these are still single-valued examples, whereas the interest of the
Definition~\ref{DefOmCont} is to authorize multiple-valuedness when following the
analytic continuation.
Elementary examples of multiple-valued continuation are provided by
$\sum_{n\ge1} \frac{\ze^n}{n} = - \log(1-\ze)$ (principal branch of the
logarithm), which is $\Om$-continuable if and only if $1\in\Om$,
and $\sum_{n\ge0} \frac{\ze^n}{n+1} = - \frac{1}{\ze}\log(1-\ze)$, 
which is $\Om$-continuable if and only if $\{0,1\}\subset\Om$.
\end{exa}

\begin{exa}
If $\om\in\C^*$ and $m\in\N^*$, then
$\big(\log(\ze-\om)\big)^m \in \hat\gR_{\{\om\}}$;
if moreover $\om\neq-1$, then
$\big(\log(\ze-\om)\big)^{-m}\in \hat\gR_{\{\om,\om+1\}}$.
\end{exa}

\begin{exa}
If $\Om$ is a closed discrete subset of~$\C$, $0\notin\Om$, $\om\in\Om$ and
$\hat\psi$ is holomorphic in $\C\setminus\Om$, then $\hat\ph(\ze) =
\hat\psi(\ze)\log(\ze-\om)$ defines a germ of $\hat\gR_\Om$ whose monodromy
around~$\om$ is given by $2\pi\I\hat\psi$.
\end{exa}


\begin{nota}	\label{rempathscont}
Given a path $\ga \col [a,b]\to\C$, if $\hat\ph$ is a holomorphic germ
at~$\ga(a)$ which admits an analytic continuation along~$\ga$, we denote by
$\cont_\ga\hat\ph$
the resulting holomorphic germ at the endpoint~$\ga(b)$.

As is often the case with analytic continuation and Cauchy integrals, the
precise parametrisation of our paths will usually not matter, in the sense that
we shall get the same results from two paths
$\ga\col [a,b]\to \C$ and 
$\ga'\col [a',b']\to \C$
which only differ by a change of parametrisation ($\ga = \ga'\circ\sig$ with
$\sig\col [a,b]\to [a',b']$ piecewise continuously differentiable, increasing
and mapping $a$ to~$a'$ and $b$ to~$b'$).

We identify $\C\{\ze\}$, the space of power series with positive radius of
convergence, with the space of holomorphic germs at~$0$. 
Given $\hat\ph\in\C\{\ze\}$, we shall often denote by the same symbol~$\hat\ph$ 
the holomorphic function it defines, or even the principal branch of its
analytic continuation when such a notion is well-defined.
\end{nota}


\section{The convolution of holomorphic germs at the origin}	\label{sec_contconvoleasy}


The convolution in $\C\{\ze\}$ is defined by the formula
\[
\hat\ph*\hat\psi(\ze) \defeq \int_0^\ze \hat\ph(\xi)\hat\psi(\ze-\xi) \,\dd\xi
\]
for any $\hat\ph,\hat\psi\in\C\{\ze\}$: the formula makes sense for $|\ze|$
small enough and defines a holomorphic germ at~$0$ whose disc of convergence
contains the intersection of the discs of convergence of~$\hat\ph$ and~$\hat\psi$.
The convolution law~$*$ is commutative and associative.%
\footnote{Indeed, the formal Borel transform
$\ti\ph(z) = \sum a_n z^{-n-1} \mapsto \hat\ph(\ze) = \sum a_n \frac{\ze^n}{n!}$
turns the Cauchy product of $z\ii\C[[z\ii]]$ into convolution 
(and the Laplace transform
$(\cL\hat\ph)(z) \defeq \int_0^\infty \ee^{-z\ze} \hat\ph(\ze)\,\dd\ze$
turns the convolution into the ordinary product of analytic functions).}

The question we address in this article is the question of the stability
of~$\hat\gR_\Om$ under convolution.
As already mentioned, this is relevant when dealing with the formal solutions of
nonlinear problems and this is absolutely necessary to develop the theory of
resurgent functions and alien calculus for $\Om$-continuable germs.

This amounts to inquiring about the analytic continuation of the germ
$\hat\ph*\hat\psi$ when $\Om$-continuability is assumed for~$\hat\ph$
and~$\hat\psi$.
Let us first mention an easy case, which is used in \cite{EVinv} and \cite{mouldSN}:
\begin{lemma}	\label{lemeasyconvol}
Let $\Om$ be any non-empty closed discrete subset of~$\C$ and suppose $\hat A$ is an entire
function of~$\C$.
Then, for any $\hat\ph\in\hat\gR_\Om$, the convolution product $\hat A*\hat\ph$
belongs to~$\hat\gR_\Om$;
its analytic continuation along a path~$\ga$ of $\C\setminus\Om$ starting
from a point~$\ze_0$ close enough to~$0$ and ending at a point~$\ze_1$ is the holomorphic germ
at~$\ze_1$ explicitly given by
\beglabel{eqcontAph}
\cont_\ga(\hat A*\hat\ph)(\ze) = 
\int_0^{\ze_0} \hat A(\ze-\xi) \hat\ph(\xi) \,\dd\xi 
+ \int_\ga \hat A(\ze-\xi) \hat\ph(\xi) \,\dd\xi 
+ \int_{\ze_1}^\ze \hat A(\ze-\xi) \hat\ph(\xi) \,\dd\xi
\elabel
for $\ze$ close enough to~$\ze_1$.
\end{lemma}


The proof is left as an exercise
(see \eg the proof of Lemma~\ref{lemcontsymOm} for a formalized proof
in a more complicated situation), 
but we wish to emphasize that formulas such
as~\eqref{eqcontAph} require a word of caution:
the value of $\hat A(\ze-\xi)$ is unambiguously defined whatever~$\ze$ and~$\xi$
are, but in the notation ``$\hat\ph(\xi)$'' it is understood that we are using the
appropriate branch of the possibily multiple-valued function~$\hat\ph$;
in such a formula, what branch we are using is clear from the context: 
\begin{enumerate}[$-$]
\item $\hat\ph$ is unambiguously defined in its disc of convergence~$D_0$ (centred
at~$0$) and the first integral thus makes sense for $\ze_0\in D_0$;
\item in the second integral
$\xi$ is moving along~$\ga$ which is a path of analytic continuation for~$\hat\ph$,
we thus consider the analytic continuation of~$\hat\ph$ along the piece of~$\ga$
between its origin and~$\xi$;
\item in the third integral, ``$\hat\ph$'' is to be understood as $\cont_\ga\hat\ph$, the germ at~$\ze_1$
resulting form the analytic continuation of~$\hat\ph$ along~$\ga$, this
integral then makes sense for any $\ze$ at a distance from~$\ze_1$ less than the radius of
convergence of $\cont_\ga\hat\ph$.
\end{enumerate}

Using a parametrisation $\ga\col[0,1]\to \C\setminus\Om$, with $\ga(0)=\ze_0$
and $\ga(1)=\ze_1$, and introducing the truncated paths
$\ga_s \defeq \ga_{|[0,s]}$ for any $s\in[0,1]$, the interpretation of the last two
integrals in~\eqref{eqcontAph} is
\begin{align*}
\int_\ga \hat A(\ze-\xi) \hat\ph(\xi) \,\dd\xi &\defeq
\int_0^1 \hat A(\ze-\ga(s)) (\cont_{\ga_s} \hat\ph)(\ga(s)) \ga'(s)\,\dd s,\\[1ex]
\int_{\ze_1}^\ze \hat A(\ze-\xi) \hat\ph(\xi) \,\dd\xi &\defeq
\int_{\ze_1}^\ze \hat A(\ze-\xi) (\cont_\ga \hat\ph)(\xi) \,\dd\xi.
\end{align*}
%


\section{Main result}	\label{sec_contconvolgen}


We now wish to be able to consider the convolution of two $\Om$-continuable
holomorphic germs at~$0$ without assuming that any of them extends to an entire
function.
The main result of this article is


\begin{thm}	\label{thmOmstbgROmstb}
Let $\Om$ be a non-empty closed discrete subset of~$\C$.
Then the space $\hat\gR_\Om$ is stable under convolution if and only if $\Om$ is stable under
addition. 
\end{thm}


The necessary and sufficient condition on~$\Om$ is satisfied by the typical
examples~$\Z$ or $2\pi\I\Z$, but also by~$\N^*$, $\Z+\I\Z$, $\N^*+\I\N$ or
$\{ m+n\sqrt{2} \mid m,n\in\N^* \}$ for instance.

The rest of the article is dedicated to the proof of Theorem~\ref{thmOmstbgROmstb}.
The necessity of the condition on~$\Om$ will follow from the following elementary
example:


\begin{exa}[\cite{CNP}]	\label{exanohope}
Let us consider $\om_1,\om_2\in\C^*$, $\hat\ph_1(\ze) = \frac{1}{\ze-\om_1}$,
$\hat\ph_2(\ze) = \frac{1}{\ze-\om_2}$ and study
\[
\hat\chi(\ze) = \hat\ph_1*\hat\ph_2(\ze) = 
\int_0^\ze \frac{1}{(\xi-\om_1)(\ze-\xi-\om_2)} \,\dd\xi,
\qquad |\ze| < \min\big\{ |\om_1|,|\om_2| \big\}.
\]
The formula
\[
\frac{1}{(\xi-\om_1)(\ze-\xi-\om_2)}= \frac{1}{\ze-\om_1-\om_2} \left(
\frac{1}{\xi-\om_1} + \frac{1}{\ze-\xi-\om_2} 
\right)
\]
shows that, for any $\ze\neq\om_1+\om_2$ of modulus $< \min\big\{
|\om_1|,|\om_2| \big\}$, one can write
\begin{equation}
\hat\chi(\ze) = \frac{1}{\ze-\om_1-\om_2} \big( L_1(\ze) + L_2(\ze) \big),
\qquad L_j(\ze) \defeq \int_0^\ze \frac{\dd\xi}{\xi-\om_j}
\end{equation}
(with the help of the change of variable $\xi \mapsto \ze-\xi$ in the case of~$L_2$).

Removing the half-lines $\om_j[1,+\infty)$ from~$\C$, we obtain a cut
plane~$\De$ in which~$\hat\chi$ has a meromorphic continuation (since $[0,\ze]$ avoids the
points~$\om_1$ and~$\om_2$ for all $\ze\in\De$). 
We can in fact follow the meromorphic continuation of~$\hat\chi$ along any path
which avoids~$\om_1$ and~$\om_2$, because 
\[
L_j(\ze) = -\int_0^{\ze/\om_j} \frac{\dd\xi}{1-\xi} =
\log\Big( 1-\frac{\ze}{\om_j} \Big) \in \hat\gR_{\{\om_j\}}.
\]
We used the words ``meromorphic continuation'' and not ``analytic continuation''
because of the factor $\frac{1}{\ze-\om_1-\om_2}$.
The conclusion is thus only $\hat\chi \in \hat\gR_\Om$, with $\Om \defeq \{\om_1,\om_2,\om_1+\om_2\}$.

\medskip

-- If $\om\defeq\om_1+\om_2\in \De$, the principal branch of~$\hat\chi$ (\ie its
meromorphic continuation to~$\De$) has a removable singularity%
\footnote{This is consistent with the well-known fact that the space of holomorphic
functions of an open set~$\De$ which is star-shaped \wrt~$0$ is stable under convolution.}
at $\om$, because 
$(L_1+L_2)(\om) = \int_0^\om \frac{\dd\xi}{\xi-\om_1} +
\int_0^\om \frac{\dd\xi}{\xi-\om_2} = 0$
in that case (by the change of variable $\xi \mapsto \om-\xi$ in one of the
integrals).
But it is easy to see that this does not happen for all the branches
of~$\hat\chi$: when considering all the paths~$\ga$ going from~$0$ to~$\om$ and
avoiding~$\om_1$ and~$\om_2$, we have
\[
\cont_\ga L_j(\om) = \int_\ga \frac{\dd\xi}{\xi-\om_j}, \qquad j=1,2,
\]
hence $\frac{1}{2\pi\I}\big(\cont_\ga L_1(\om) + \cont_\ga L_2(\om)\big)$ is the
sum of the winding numbers around~$\om_1$ and~$\om_2$ of the loop obtained by
concatenating~$\ga$ and the line segment $[\om,0]$;
elementary geometry shows that this sum of winding numbers can take any
integer value, but whenever this value is non-zero the corresponding
branch of~$\hat\chi$ does have a pole at~$\om$.

\medskip

-- The case $\om\notin\De$ is slightly different. Then we can write 
$\om_j = r_j\,\ee^{\I\th}$ with $r_1,r_2>0$ 
and consider the path~$\ga_0$ which follows the segment $[0,\om]$ except that it
circumvents $\om_1$ and~$\om_2$ by small half-circles travelled anti-clockwise
(notice that $\om_1$ and~$\om_2$ may coincide);
an easy computation yields
\[
\cont_{\ga_0} L_1(\om) = \int_{-r_1}^{-1} \frac{\dd\xi}{\xi}
+ \int_1^{r_2} \frac{\dd\xi}{\xi}
+ \int_{\Ga_0} \frac{\dd\xi}{\xi},
\]
where $\Ga_0$ is the half-circle from~$-1$ to~$1$ with radius~$1$ travelled
anti-clockwise,
hence $\cont_{\ga_0} L_1(\om) = \log\frac{r_2}{r_1} + \I\pi$, 
similarly $\cont_{\ga_0} L_2(\om) = \log\frac{r_1}{r_2} + \I\pi$, 
therefore $\cont_{\ga_0} L_1(\om) + \cont_{\ga_0} L_2(\om) = 2\pi\I$ is non-zero and this
again yields a branch of~$\hat\chi$ with a pole at~$\om$ (and infinitely many
others by using other paths than~$\ga_0$).

\medskip

In all cases, there are paths from~$0$ to $\om_1+\om_2$ which avoid~$\om_1$
and~$\om_2$ and which are not paths of analytic continuation for~$\hat\chi$.
This example thus shows that $\hat\gR_{\{\om_1,\om_2\}}$ is \emph{not} stable under
convolution: it contains $\hat\ph_1$ and~$\hat\ph_2$ but not $\hat\ph_1*\hat\ph_2$.
\end{exa}

Now we see that for $\hat\gR_\Om$ to be stable under convolution it is necessary
that $\Om$ be stable under addition: if not, one can find
$\om_1,\om_2\in\Om$ such that $\om_1+\om_2\notin \Om$ and
Example~\ref{exanohope} then yields $\hat\ph_1,\hat\ph_2\in\hat\gR_\Om$ with
$\hat\ph_1*\hat\ph_2 \notin \hat\gR_\Om$.
This gives the easy part of Theorem~\ref{thmOmstbgROmstb}.


\section{Proof of the main result: Analytic part}


From now on we assume that~$\Om$ is stable under addition.
Our aim is to prove that this is sufficient to entail the stability under
convolution of~$\hat\gR_\Om$.
We begin with a definition, illustrated by Figure~\ref{fig:SymHom}:
\begin{Def}	\label{defSymOmHom}
A continuous map
$H \col I \times J \to \C$, where $I=[0,1]$ and $J$ is a compact interval
of~$\R$, is called a symmetric $\Om$-homotopy if, for each $t \in J$,
\[ s \in I \mapsto H_t(s) \defeq H(s,t) \]
defines a path which satisfies 
\begin{enumerate}
\item $H_t(0) = 0$,
\item $H_t\big( (0,1] \big) \subset \C\setminus\Om$,
\item $H_t(1) - H_t(s) = H_t(1-s)$ for every $s \in I$.
\end{enumerate}
We then call endpoint path of~$H$ the path
\[
\Ga_H \col t \in J \mapsto H_t(1).
\]
Writing $J=[a,b]$, we call $H_a$ (\resp $H_b$) the initial path of~$H$ (\resp its final path).
\end{Def}


\begin{figure}
\begin{center}
\epsfig{file=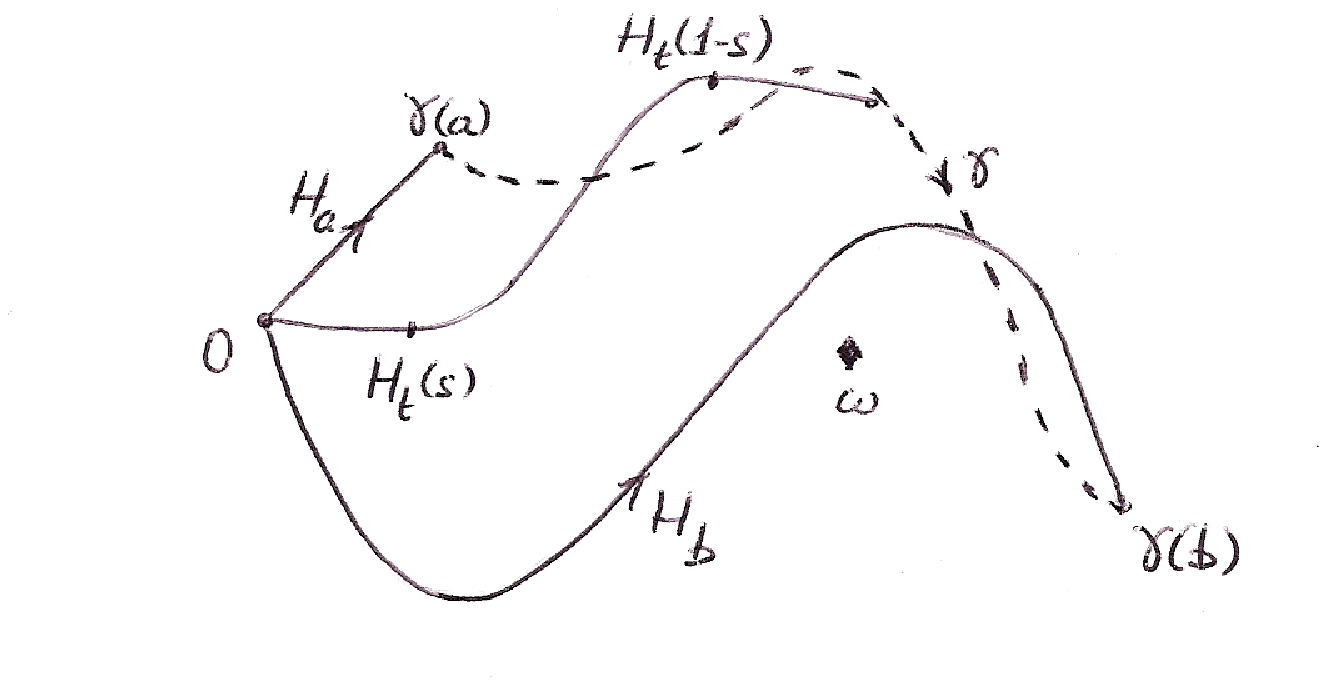,height=2.2in,angle = 0}
\vspace{-.55cm}

\caption{%
A symmetric $\Om$-homotopy, with its initial path $H_a$, its final path $H_b$
and its endpointpath $\ga = \Ga_H$.}
\label{fig:SymHom}
\end{center}
\end{figure}

The first two conditions imply that each path~$H_t$ is a path of analytic
continuation for any $\hat\ph\in\hat\gR_\Om$, in view of
Remark~\ref{reminitialpoint}.

We shall use the notation $H_{t|s}$ for the truncated paths $(H_t)_{|[0,s]}$,
$s\in I$, $t\in J$ (analogously to what we did when commenting
Lemma~\ref{lemeasyconvol}).
Here is a technical statement we shall use:
\begin{lemma}	\label{lemtechnic}
For a symmetric $\Om$-homotopy~$H$ defined on $I\times J$,
there exists $\de>0$ such that, for any
$\hat\ph\in\hat\gR_\Om$ and $(s,t) \in I\times J$, the radius of convergence of the
holomorphic germ $\cont_{H_{t|s}}\hat\ph$ at~$H_t(s)$ is at least~$\de$.
\end{lemma}


\begin{proof}
Let $\rho$ be as in Remark~\ref{reminitialpoint}.
Consider %
\[
U\defeq \big\{\, (s,t) \in I\times J \mid H\big( [0,s]\times\{t\} \big) 
\subset\Ddem \,\},
\qquad K \defeq I\times J \setminus U.
\]
Writing $K = 
\big\{\, (s,t) \in I\times J \mid \exists s'\in[0,s] \;\text{s.t.}\;
H(s',t) \in \C\setminus\Ddem \,\}$, 
we see that $K$ is a compact subset of $I\times J$ which is contained in
$(0,1]\times J$.
Thus $H(K)$ is a compact subset of $\C\setminus\Om$,
and $\de \defeq \min\big\{ \dist\big( H(K), \Om \big), \rho/2 \big\} > 0$.
Now, for any $s$ and~$t$, 
\begin{enumerate}[--]
\item either $(s,t)\in U$, then 
the truncated path~$H_{t|s}$ lies in~$\Ddem$, hence
$\cont_{H_{t|s}}\hat\ph$ is a holomorphic germ
at~$H_t(s)$ with radius of convergence $\ge \de$;
\item or $(s,t)\in K$, and then $\dist(H_t(s),\Om)\ge\de$, which yields the same
conclusion for the germ $\cont_{H_{t|s}}\hat\ph$.
\end{enumerate}
\end{proof}


The third condition in Definition~\ref{defSymOmHom} means that each path~$H_t$
is symmetric \wrt\ its midpoint~$\demi H_t(1)$.
Here is the motivation behind this requirement:
\begin{lemma}	\label{lemcontsymOm}
Let $\ga \col [0,1] \to \C \setminus \Om$ be a path such that $\ga(0) \in
\D_\rho$, with~$\rho$ as in Remark~\ref{reminitialpoint}.
If there exists a symmetric $\Om$-homotopy whose endpoint path coincides
with~$\ga$ and whose initial path is contained in~$\D_\rho$, 
then any convolution product $\hat\ph*\hat\psi$ with $\hat\ph,
\hat\psi\in\hat\gR_\Om$ can be analytically continued along~$\ga$.
\end{lemma}

\begin{proof}
We assume that $H$ is defined on $I\times J$ and we set $\ga \defeq \Ga_H$.
Let $\hat\ph, \hat\psi\in\hat\gR_\Om$ and, for $t \in J$, consider the formula
\beglabel{eqpremform}
\hat\chi_t(\ze) = \int_{H_t} \hat\ph(\xi) \hat\psi(\ze-\xi)\,\dd\xi
+ \int_{\ga(t)}^\ze \hat\ph(\xi) \hat\psi(\ze-\xi)\,\dd\xi
\elabel
(recall that $\ga(t) = H_t(1)$). We shall check that $\hat\chi_t$ is a
well-defined holomorphic germ at~$\ga(t)$ and that it provides the analytic
continuation of $\hat\ph*\hat\psi$ along~$\ga$.

\medskip

\noindent a) 
The idea is that when $\xi$ moves along~$H_t$, $\xi = H_t(s)$ with $s \in I$,
we can use for ``$\hat\ph(\xi)$'' the analytic continuation of~$\hat\ph$
along the truncated path~$H_{t|s}$;
correspondingly, if~$\ze$ is close to~$\ga(t)$, then $\ze-\xi$ is close to
$\ga(t)-\xi = H_t(1) - H_t(s) = H_t(1-s)$, thus for
``$\hat\psi(\ze-\xi)$'' we can use the analytic continuation of~$\hat\psi$
along~$H_{t|1-s}$.
In other words, setting $\ze=\ga(t)+\sig$, we wish to interpret~\eqref{eqpremform} as
\begin{multline}	\label{eqsecform}
\hat\chi_t(\ga(t)+\sig) \defeq \int_0^1 
(\cont_{H_{t|s}}\hat\ph)(H_t(s)) (\cont_{H_{t|1-s}}\hat\psi)(H_t(1-s)+\sig)
H_t'(s)\,\dd s \\
+ \int_0^1 (\cont_{H_t}\hat\ph)(\ga(t)+u\sig) \hat\psi((1-u)\sig)\sig\,\dd u
\end{multline}
(in the last integral, we have performed the change variable $\xi =
\ga(t)+u\sig$; it is the germ of~$\hat\psi$ at the origin that we use
there).

Lemma~\ref{lemtechnic} provides $\de>0$ such that, by regular dependence of the
integrals upon the parameter~$\sig$, the \rhs\ of~\eqref{eqsecform} is
holomorphic for $|\sig|<\de$.
We thus have a family of analytic elements $(\hat\chi_t,D_t)$, $t \in J$, with
$D_t \defeq D\big(\ga(t), \de\big)$.

\medskip

\noindent b) 
For $t$ small enough, the path $H_t$ is contained in~$\D_\rho$ which is open and simply
connected; then, for $|\ze|$ small enough, the line segment $[0,\ze]$ and
the concatenation of~$H_t$ and $[\ga(t),\ze]$ are homotopic in~$\D_\rho$, hence the
Cauchy theorem implies $\hat\chi_t(\ze) = \hat\ph*\hat\psi(\ze)$.

\medskip

\noindent c)
By uniform continuity, there exists $\eps>0$ such that, for any $t_0,t \in J$,
\beglabel{inequnifcont}
|t-t_0| \le \eps 
\quad\Longrightarrow\quad
|H_t(s)-H_{t_0}(s)| < \de/2 \quad \text{for all $s \in I$}.
\elabel
To complete the proof, we check that, for any $t_0, t$ in~$J$ such that
$t_0\le t \le t_0+\eps$, we have $\hat\chi_{t_0} \equiv \hat\chi_{t}$ in
$D\big(\ga(t_0),\de/2)$ (which is contained in $D_{t_0}\cap D_{t}$).

Let $t_0,t\in J$ be such that $t_0\le t \le t_0+\eps$ and let $\ze \in D\big(\ga(t_0),\de/2)$.
By Lemma~\ref{lemtechnic} and~\eqref{inequnifcont}, we have for every $s \in I$
\begin{align*}
&\cont_{H_{t|s}}\hat\ph\big(H_t(s)\big) = \cont_{H_{t_0|s}}\hat\ph\big(H_t(s)\big), \\
&\cont_{H_{t|1-s}}\hat\psi\big(\ze-H_t(s)\big) =
\cont_{H_{t_0|1-s}}\hat\psi\big(\ze-H_t(s)\big)
\end{align*}
(for the latter identity, write
$\ze-H_t(s) = H_t(1-s) + \ze - \ga(t) = H_{t_0}(1-s) + \ze-\ga(t_0) +
H_{t_0}(s)-H_t(s)$, thus this point belongs to $D\big(H_t(1-s),\de)
\cap D\big(H_{t_0}(1-s),\de)$).
Moreover, $[\ga(t),\ze] \subset D\big(\ga(t_0),\de/2)$ by convexity,
hence $\cont_{H_{t}}\hat\ph \equiv \cont_{H_{t_0}}\hat\ph$ on this line segment,
and we can write
\begin{multline*}	
\hat\chi_t(\ze) = \int_0^1 
(\cont_{H_{t_0|s}}\hat\ph)(H_t(s)) (\cont_{H_{t_0|1-s}}\hat\psi)(\ze-H_t(s))
H_t'(s)\,\dd s \\
+ \int_{\ga(t)}^\ze (\cont_{H_{t_0}}\hat\ph)(\xi) \hat\psi(\ze-\xi)\,\dd\xi.
\end{multline*}
We then get $\hat\chi_{t_0}(\ze) = \hat\chi_{t}(\ze)$ from the Cauchy theorem by
means of the homotopy induced by~$H$ between 
the concatenation of~$H_{t_0}$ and $[\ga(t_0),\ze]$ 
and the concatenation of~$H_{t}$ and $[\ga(t),\ze]$.
\end{proof}


\begin{rem}
Definition~\ref{defSymOmHom} is not really new: when the initial path~$H_a$ is a
line segment contained in~$\D_\rho$, the final path~$H_b$ is what
\'Ecalle calls a ``symmetrically contractile path'' in \cite{Eca81}.
The proof of Lemma~\ref{lemcontsymOm} shows that the analytic continuation of
$\hat\ph*\hat\psi$ until the endpoint $H_b(1)=\Ga_H(b)$ can be computed by the usual
integral taken over~$H_b$ (however, it usually cannot be computed as the same
integral over the endpoint path~$\Ga_H$, even when the latter integral is
well-defined).
\end{rem}


\section{Proof of the main result: Geometric part}


\subsection{The key lemma}


In view of Lemma~\ref{lemcontsymOm}, the proof of Theorem~\ref{thmOmstbgROmstb}
will be complete if we prove the following purely geometric result:

\begin{lemma}	\label{lemkey}
For any path $\ga \col I= [0,1]\to \C\setminus\Om$ such that $\ga(0)\in \D^*_\rho$
and the left and right derivatives $\ga'_\pm$ do not vanish on~$I$,
there exists a symmetric $\Om$-homotopy~$H$ on $I\times I$ whose endpoint path is~$\ga$ and
whose initial path is a line segment, \ie $\Ga_H=\ga$ and $H_0(s)\equiv s\ga(0)$.
\end{lemma}

The proof is strikingly simple when $\ga$ does not pass through~$0$, which is
automatic if we assume $0\in\Om$.
The general case requires an extra work which is technical and involves a
quantitative version of the simpler case.
With a view to helping the reader to grasp the mechanism of the proof, we thus begin
with the case when $0\in\Om$.


\subsection{Proof of the key lemma when $0\in\Om$}	\label{seczeroinOm}


Assume that $\ga$ is given as in the hypothesis of Lemma~\ref{lemkey}.
We are looking for a symmetric $\Om$-homotopy whose initial path is imposed: it
must be
\[
s \in I \mapsto H_0(s) \defeq s\ga(0),
\]
which satisfies the three requirements of Definition~\ref{defSymOmHom} at $t=0$:
\begin{enumerate}[(i)]
\item
$H_0(0)=0$, 
\item
$H_0\big( (0,1] \big) \subset \C\setminus \Om$, 
\item
$H_0(1) - H_0(s) = H_0(1-s)$ for every $s \in I$.
\end{enumerate}
The idea is to define a family of maps $(\Psi_t)_{t\in[0,1]}$ so that
\beglabel{eqdefHsPsis}
H_t(s) \defeq \Psi_t\big( H_0(s) \big), \qquad s \in I,
\elabel
yield the desired homotopy. For that, it is sufficient that $(t,\ze) \in[0,1]
\times \C \mapsto \Psi_t(\ze)$ be continuously differentiable (for the structure
of real two-dimensional vector space of~$\C$), $\Psi_0=\ID$ and, for each $t\in
[0,1]$,
\begin{enumerate}[(i')]
\item
$\Psi_t(0) = 0$,
\item
$\Psi_t(\C\setminus\Om) \subset \C\setminus\Om$,
\item
$\Psi_t\big( \ga(0) - \ze \big) = \Psi_t\big( \ga(0) \big) - \Psi_t(\ze)$ for
all $\ze\in\C$,
\item
$\Psi_t\big(\ga(0)\big) = \ga(t)$.
\end{enumerate}
In fact, the properties (i')--(iv') ensure that any initial path $H_0$
satisfying (i)--(iii) and ending at~$\ga(0)$ produces through~\eqref{eqdefHsPsis}
a symmetric $\Om$-homotopy whose endpoint path is~$\ga$. Consequently, we may
assume without loss of generality that~$\ga$ is~$C^1$ on $[0,1]$ 
(then, if $\ga$ is only piecewise~$C^1$, we just need to concatenate the
symmetric $\Om$-homotopies associated with the various pieces).

The maps~$\Psi_t$ will be generated by the flow of a non-autonomous vector
field $X(\ze,t)$ associated with~$\ga$ that we now define.
We view $(\C,|\,\cdot\,|)$ as a real $2$-dimensional Banach space and pick%
\footnote{
For instance pick a $C^1$ function $\ph_0 \col \R\to[0,1]$ such that 
$\{\, x\in\R \mid \ph_0(x) = 1 \,\} = \{0\}$ and $\ph_0(x)=0$ for $|x|\ge1$,
and a bijection $\om\col\N\to\Om$;
then set $\de_k \defeq \dist\big( \om(k), \Om\setminus\{\om(k)\} \big) >0$ and
$\sig(\ze) \defeq \sum_k \ph_0\big( \frac{4|\ze-\om(k)|^2}{\de_k^2} \big)$:
for each $\ze\in\C$ there is at most one non-zero term in this series 
(because $k\neq\ell$, $|\ze-\om(k)| < \de_k/2$ and $|\ze-\om(\ell)| <
\de_\ell/2$ would imply $|\om(k)-\om(\ell)|< (\de_k+\de_\ell)/2$, which would
contradict $|\om(k)-\om(\ell)| \ge \de_k$ and~$\de_\ell$),
thus $\sig$ is $C^1$, takes its values in $[0,1]$ and satisfies
$\{\, \ze\in\C \mid \sig(\ze) = 1 \,\} = \Om$, therefore $\eta\defeq 1-\sig$
will do.
Other solution: adapt the proof of Lemma~\ref{lemetatech}.
}
a $C^1$ function $\eta \col \C \to [0,1]$ such that
\[
\{\,\ze\in\C \mid \eta(\ze) = 0 \,\} = \Om.
\]
Observe that $D(\ze,t) \defeq \eta(\ze) + \eta\big( \ga(t)-\ze \big)$ defines
a $C^1$ function of $(\ze,t)$
which satisfies
\[
D(\ze,t) > 0 \quad \text{for all $\ze\in\C$ and $t\in[0,1]$}
\]
\emph{because $\Om$ is stable under addition};
indeed, $D(\ze,t) = 0$ would imply $\ze\in\Om$ and $\ga(t)-\ze\in\Om$, hence
$\ga(t)\in\Om$, which would contradict our assumptions.
Therefore, the formula
\begin{equation}	\label{eqdefXnonaut}
X(\ze,t) \defeq \frac{\eta(\ze)}{\eta(\ze) + \eta\big( \ga(t)-\ze \big)} \ga'(t)
\end{equation}
defines a non-autonomous vector field, which is continuous in $(\ze,t)$ on
$\C\times[0,1]$, $C^1$ in~$\ze$ and has its partial derivatives continuous in
$(\ze,t)$.
The Cauchy-Lipschitz theorem on the existence and uniqueness of solutions to
differential equations applies to $\frac{\dd\ze}{\dd t} = X(\ze,t)$:
for every $\ze\in\C$ and $t_0\in[0,1]$ there is a unique solution $t\mapsto \Phi^{t_0,t}(\ze)$
such that $\Phi^{t_0,t_0}(\ze)=\ze$. 
The fact that the vector field~$X$ is bounded implies that $\Phi^{t_0,t}(\ze)$ is
defined for all $t\in[0,1]$
and the classical theory guarantees that $(t_0,t,\ze)\mapsto\Phi^{t_0,t}(\ze)$
is $C^1$ on $[0,1]\times[0,1]\times\C$.

Let us set $\Psi_t \defeq \Phi^{0,t}$ for $t\in[0,1]$ and check that this family of
maps satisfies (i')--(iv').
We have
\begin{gather}
\label{eqvanishX}
X(\om,t) = 0 \quad\text{for all $\om\in\Om$,}\\
\label{eqsymX}
X\big( \ga(t)-\ze, t \big) = \ga'(t) - X(\ze,t)
\quad\text{for all $\ze\in\C$}
\end{gather}
for all $t\in[0,1]$
(by the very definition of~$X$). Therefore
\begin{itemize}
\item
(i') and (ii') follow from~\eqref{eqvanishX} which yields
$\Phi^{t_0,t}(\om)=\om$ for every $t_0$ and~$t$, whence $\Psi_t(0)=0$ since
$0\in\Om$,
and from the non-autonomous flow property $\Phi^{t,0}\circ\Phi^{0,t}=\ID$
(hence $\Psi_t(\ze)=\om$ implies $\ze=\Phi^{t,0}(\om)=\om$);
\item
(iv') follows from the fact that $X\big( \ga(t),t \big)=\ga'(t)$,
by~\eqref{eqvanishX} and~\eqref{eqsymX} with $\ze=0$, using again that
$0\in\Om$, hence $t\mapsto \ga(t)$ is a solution of~$X$;
\item
(iii') follows from~\eqref{eqsymX}: for any solution $t\mapsto\ze(t)$, the curve
$t\mapsto \xi(t)\defeq \ga(t)-\ze(t)$ satisfies $\xi(0) = \ga(0)-\ze(0)$ and
$\xi'(t) = \ga'(t) - X\big( \ze(t),t \big) = X\big( \xi(t),t \big)$,
hence it is a solution:
$\xi(t) = \Psi_t\big(\ga(0)-\ze(0)\big)$.
\end{itemize}

As explained above, formula~\eqref{eqdefHsPsis} thus produces the desired
symmetric $\Om$-homotopy.

\begin{rem}
Our proof of Lemma~\ref{lemkey}, which essentially relies on the use of the flow
of the non-autonomous vector field~\eqref{eqdefXnonaut}, arose as an attempt to
understand a related but more complicated construction which can be found in an
appendix of the book \cite{CNP}
(however the vector field there was autonomous and we must confess that we were
not able to follow completely the arguments of \cite{CNP}).
\end{rem}


\subsection{Proof of the key lemma when $0\notin\Om$}


From now on, we suppose $0\notin\Om$ and we use the notation
\[
\Om_\eps \defeq \{\, \ze \in \C \mid \dist(\ze,\Om) < \eps \,\}
\]
for any $\eps>0$,
hence $\ov\Om_{\eps}= \{\, \ze \in \C \mid \dist(\ze,\Om) \le \eps \,\}$.
We shall require the following technical


\begin{lemma}	\label{lemetatech}
For any $\eps>0$ there exists a $C^1$ function $\eta \col \C \to [0,1]$ such
that
\[
\{\, \ze\in\C \mid \eta(\ze) = 0 \,\} = \{0\} \cup \ov\Om_{\eps}.
\]
\end{lemma}

\begin{proof}
Pick a $C^1$ function $\chi \col \R \to [0,1]$ such that
$\{\, x\in\R \mid \chi(x) = 0 \,\} = [-\eps^2,\eps^2]$ and $\chi(x)=1$ for
$|x|\ge (1+\eps)^2$,
and a bijection $\om\col\N^*\to\Om$.
For each $k\in\N^*$, 
$
\eta_k(\ze) \defeq \chi\big( |\ze-\om(k)|^2 \big)
$
defines a $C^1$ function on~$\C$ such that $\eta_k\ii(0) = \ov{D(\om(k),\eps)}$ and
$\eta_k\equiv 1$ on $\C\setminus D(\om(k),1+\eps)$.
Consider the infinite product
\beglabel{eqdefetaprod}
\eta_*(\ze) \defeq \prod_{k\in\N^*} \eta_k(\ze).
\elabel
For any bounded open subset~$U$ of~$\C$, the set $\cF_U \defeq \{\, k\in\N^* \mid
U\cap D(\om(k),1+\eps) \neq\emptyset \,\}$ is finite (because $\Om$ is discrete), thus almost all the
factors in~\eqref{eqdefetaprod} are equal to~$1$ when $\ze\in U$:
$(\eta_*)_{|U} = \prod_{k\in\cF_U} (\eta_k)_{|U}$, hence $\eta_*$ is~$C^1$, takes its values
in~$[0,1]$ and 
\[
\eta_*\ii(0)\cap U = \bigcup_{k\in\cF_U} \ov{D(\om(k),\eps)} \cap U,
\]
whence it follows that $\eta_*\ii(0) = \ov\Om_{\eps}$.

If $0\in\ov\Om_\eps$, then one can take $\eta=\eta_*$. If not, then one can take
the product $\eta = \eta_0 \eta_*$ with 
$\eta_0(\ze) \defeq \chi_0(|\ze|^2)$,
where $\chi_0$ is any $C^1$ function on~$\R$ which takes its values in $[0,1]$
and such that $\chi_0\ii(0)=\{0\}$.
\end{proof}


We  now repeat the work of the previous section replacing $\Om$ with
$\{0\}\cup\Om$, adding quantitative information (we still assume that we are
given a path which does not pass through~$0$ but we want to control the way the
corresponding symmetric $\Om$-homotopy approaches the points of~$\Om$)
and authorizing a more general initial path than a rectilinear one.


\begin{lemma}	\label{sublemkey}
Let $\de,\de'>0$ with $\de' < \de/2$.
Suppose that $J = [a,b]$ is a compact interval of~$\R$ and $\ga \col J \to \C$ is a
path such that
\[ 
0 \notin \ga(J)
\quad\text{and}\quad
\ga(J) \subset \C\setminus \Om_\de.
\]
Suppose that $h \col I \to \C$ is a $C^1$ path such that
\begin{enumerate}[(i)]
\item $h(0) = 0$,
\item $h(I) \subset \C\setminus \Om_{\de'}$,
\item
$h(1-s) = h(1) - h(s)$ for all $s\in I$,
\item $h(1) = \ga(a)$.
\end{enumerate}
Then there exists a symmetric $\Om$-homotopy~$H$ defined on $I \times J$, whose
initial path is~$h$, whose endpoint path is~$\ga$, which satisfies
$H(I\times J) \subset \C\setminus{\Om_{\de'}}$
and whose final path is~$C^1$.
\end{lemma}


\begin{proof}
We may assume without loss of generality that~$\ga$ is~$C^1$ on~$J$
(if $\ga$ is only piecewise~$C^1$, we just need to concatenate the
symmetric $\Om$-homotopies associated with the various pieces).
We shall define a family of maps $(\Psi_t)_{t\in J}$ so that
\beglabel{eqdefHtPsit}
H_t(s) \defeq \Psi_t\big( h(s) \big), \qquad s \in I,
\elabel
yield the desired homotopy. For that, it is sufficient that $(t,\ze) \in J
\times \C \mapsto \Psi_t(\ze)$ be continuously differentiable, $\Psi_0=\ID$ and,
for each $t\in J$,
\begin{enumerate}[(i')]
\item
$\Psi_t(0) = 0$,
\item
$\Psi_t(\C\setminus\Om_{\de'}) \subset \C\setminus\Om_{\de'}$,
\item
$\Psi_t\big( \ga(a) - \ze \big) = \Psi_t\big( \ga(a) \big) - \Psi_t(\ze)$ for
all $\ze\in\C$,
\item
$\Psi_t\big(\ga(a)\big) = \ga(t)$.
\end{enumerate}
As in Section~\ref{seczeroinOm}, our maps~$\Psi_t$ will be generated by a
non-autonomous vector field. 

Lemma~\ref{lemetatech} allows us to choose a $C^1$ function $\eta \col \C \to [0,1]$ such that
\[
\{\,\ze\in\C \mid \eta(\ze) = 0 \,\} = \{0\} \cup \ov\Om_{\de'}.
\]
We observe that $D(\ze,t) \defeq \eta(\ze) + \eta\big( \ga(t)-\ze \big)$ defines
a $C^1$ function of $(\ze,t)$
which satisfies
\[
D(\ze,t) > 0 \quad \text{for all $\ze\in\C$ and $t\in[0,1]$}
\]
because $\Om$ is stable under addition;
indeed, $D(\ze,t) = 0$ would imply that both $\ze$ and $\ga(t)-\ze$ lie in
$\{0\} \cup \ov\Om_{\de'}$, hence $\ga(t)\in \{0\} \cup \ov\Om_{2\de'}$, which would
contradict our assumption
$\ga(J) \subset \C \setminus \big( \{0\} \cup \Om_\de \big)$.
Therefore the formula
\[
X(\ze,t) \defeq \frac{\eta(\ze)}{\eta(\ze) + \eta\big( \ga(t)-\ze \big)}
\ga'(t),
\qquad (\ze,t)\in\C\times J,
\]
defines a non-autonomous vector field whose flow $(\Phi^{t_0,t})_{t_0,t\in J}$
allows one to conclude the proof exactly as in Section~\ref{seczeroinOm},
setting $\Psi_t \defeq \Phi^{a,t}$ and
replacing~\eqref{eqvanishX} with
\[
X(\om,t) = 0 \quad\text{for all $\om\in\{0\}\cup \Om_{\de'}$.}
\]
\end{proof}


We now consider the case of a path~$\ga$ which entirely lies close to~$0$.

\begin{lemma}	\label{sublemextra}
Let $\eps,\de'>0$ with $0 < \eps < \de'$.
Suppose that $K = [a,b]$ is a compact interval of~$\R$ and $\ga \col K \to \C$ is a
path such that
\[ 
\ga(K) \subset \ov\D_{\eps/2}.
\]
Suppose that $h \col I \to \C$ is a $C^1$ path such that
\begin{enumerate}[(i)]
\item $h(0) = 0$,
\item $h(I) \subset \C\setminus \Om_{\de'}$,
\item
$h(1-s) = h(1) - h(s)$ for all $s\in I$,
\item $h(1) = \ga(a)$.
\end{enumerate}
Then there exists a symmetric $\Om$-homotopy~$H$ defined on $I \times K$, whose
initial path is~$h$, whose endpoint path is~$\ga$, which satisfies
\[
H(I\times K) \subset \C\setminus{\Om_{\de''}}
\quad\text{ with $\de'' \defeq \de'-\eps$ }
\]
and whose final path is~$C^1$.
\end{lemma}


\begin{proof}
Define $H(s,t) \defeq h(s) + s\big( \ga(t) - \ga(a) \big)$.
This way $H(s,a) = h(s)$, $H(1,t) = \ga(t)$ and $H$ is a symmetric
$\Om$-homotopy as required:
$H(0,t) = 0$,
$H(s,t) + H(1-s,t) = h(s) + h(1-s) + \ga(t) - \ga(a) = \ga(t)$,
$\dist\big( H(s,t),\Om \big) \ge \dist\big( h(s), \Om) - |\ga(t)-\ga(a)| \ge
\de' - \eps$.
\end{proof}


\medskip

\noindent
\textbf{Proof of the key lemma when $0\notin\Om$.}
Let $\ga$ be as in the hypothesis of Lemma~\ref{lemkey}. Without loss of
generality, we can assume $\ga(1)\neq0$ (if not, view $\ga$ as the
restriction of a path $\ti\ga \col [0,2] \to \C\setminus\Om$ such that
$\ga(2)\neq0$, with which is associated a symmetric $\Om$-homotopy~$\ti H$
defined on $I\times [0,2]$, and restrict~$\ti H$ to $I\times [0,1]$).
Let $\de \defeq \dist\big( \Om, \ga([0,1]) \big)$. 

The set $Z \defeq \{\, t \in [0,1] \mid \ga(t) = 0 \,\}$ is closed; it is also
discrete because of the non-vanishing of the derivatives of~$\ga$, thus it has
a finite cardinality $N\in\N$.
If $N=0$, then we can apply Lemma~\ref{sublemkey} with $J=[0,1]$ and $h(s)
\equiv s\ga(0)$ and the proof is complete.

From now on we suppose $N\ge1$. Let us write
\[
Z = \{t_1, \ldots, t_N \} \quad
\text{with $0 < t_1 < \cdots < t_N < 1$.}
\]
We define
\[
\de_0 \defeq \demi \min\Big\{ \frac{\de}{2}, \rho-|\ga(0)| \Big\}
\quad \text{and} \quad
\eps \defeq \min\Big\{ |\ga(0)|, |\ga(1)|, \frac{\de_0}{N+1} \Big\}.
\]
The continuity of~$\ga$ allows us to find pairwise disjoint closed
intervals of positive lengths $K_1,\ldots,K_N$ such that
\[
t_j \in \mathring{K}_j \quad\text{and}\quad \ga(K_j) \subset \ov\D_{\eps/2},
\qquad j=1,\ldots,N.
\]
By considering the connected components of $[0,1]\setminus\bigcup K_j$ and
taking their closures, we get adjacent closed subintervals of positive lengths
of $[0,1]$,
\[
J_0, K_1, J_1, K_2, \ldots, J_{N-1}, K_N, J_N
\]
with $J_j = [a_j,b_j]$, $K_j = [b_{j-1},a_j]$, $a_0=0$, $b_N=1$.
Observe that
\[
0\notin\ga(J_j) \quad \text{and}\quad \ga(J_j) \subset \C\setminus \Om_\de,
\qquad j = 0,\ldots,N.
\]

\begin{itemize}
\item
We apply Lemma~\ref{sublemkey} with $J=J_0 = [0,b_0]$, $h(s)\equiv s\ga(0)$ and $\de' =
\de_0$ (which is allowed by the choice of~$\de_0$): 
we get a symmetric $\Om$-homotopy~$H$ defined on $I\times J_0$ whose initial
path is the line segment $[0,\ga(0)]$, whose endpoint path is
$\ga_{|J_0}$ and whose final path $H_{b_0}$ is $C^1$ and lies in $\C\setminus
\Om_{\de_0}$.
\item
We apply Lemma~\ref{sublemextra} with $K=K_1$, $\de'=\de_0$ and $h=H_{b_0}$:
we get an extension of our symmetric $\Om$-homotopy~$H$ to $I\times K_1$, in
which the enpoint path is extended by $\ga_{|K_1}$ and the final path is now
$H_{a_1}$, a $C^1$ path contained in $\C\setminus \Om_{\de_1}$
with $\de_1 \defeq \de_0 - \eps$.
\item
And so on: we apply alternatively Lemma~\ref{sublemextra} on~$K_j$ and
Lemma~\ref{sublemkey} on~$J_j$:
we get an extension of the symmetric $\Om$-homotopy~$H$ to $I\times K_j$ or
$I\times J_j$ such that
both $H_{a_j}(I)$ and $H_{b_j}(I)$ are contained in $\C\setminus \Om_{\de_j}$
with $\de_j \defeq \de_0 - j \eps$.
\end{itemize}

When we reach $j=N$, the proof of Lemma~\ref{lemkey} is complete.
\eopf




\noindent {\em Acknowledgements.}
The author wishes to thank the anonymous referee for helping to improve this article.
The research leading to these results has received funding from the
European Comunity's Seventh Framework Program (FP7/2007--2013) under Grant
Agreement n.~236346.




\vspace{1cm}

\noindent
David Sauzin\\[1ex]
CNRS UMI 3483 - Laboratorio Fibonacci \\
Collegio Puteano, Scuola Normale Superiore di Pisa \\
Piazza dei Cavalieri 3, 56126 Pisa, Italy\\
email:\,{\tt{david.sauzin@sns.it}}

\end{document}